\documentclass[12pt]{amsart}

\usepackage{amsmath,amssymb,amscd}
\usepackage{epsfig}
\usepackage{color}
\usepackage{a4wide}

\newtheorem{thm}{Theorem}[section]
\newtheorem{proposition}{Proposition}[section]
\newtheorem{prop}{Proposition}[section]
\newtheorem{lemma}{Lemma}[section]

\newtheorem{cor}{Corollary}[section]

\theoremstyle{remark}

\DeclareGraphicsExtensions{.eps}

\newcommand{\R}{\mathbb R}
\newcommand{\Z}{\mathbb Z}

\newcommand{\N}{\mathbb N}

\newcommand{\A}{{\mathcal A}}

\newcommand{\RRR}{{\mathfrak R}}

\newcommand{\EEE}{{\mathfrak E}}

\newcommand{\ra}{\rightarrow}
\newcommand{\s}{\sigma}

\title{ Symmetric intersections of Rauzy fractals} 
\date{\today}

\author[T. Sellami]{Tarek Sellami$^1$}
\author[V.F. Sirvent]{V\'{\i}ctor F. Sirvent$^2$}

\address{$^1$Sfax University, Faculty of sciences of Sfax, Department of mathematics, Route Soukra
BP 802, 3018 Sfax, Tunisia.}
\email{tarek.sellami.math@gmail.com} 

\address{$^2$Departamento de Matematicas, Universidad Sim\'{o}n
Bol\'{\i}var, Apartado 89000, Caracas 1086-A, VENEZUELA.}

\email{vsirvent@usb.ve}
\urladdr{http://www.ma.usb.ve/\textasciitilde  vsirvent}



\begin{document}

\begin{abstract}
In this article we study symmetric subsets of Rauzy fractals of unimodular irreducible Pisot substitutions.
The symmetry considered is reflection through the origin.
Given an unimodular irreducible Pisot substitution, we consider the intersection of its 
Rauzy fractal with the Rauzy fractal of the reverse substitution.
This set is symmetric and it is obtained by the balanced pair algorithm associated with both substitutions.
\end{abstract}

\maketitle

\section{Introduction}\label{sec:intro}
The Rauzy fractal is an important object in the study of dynamical systems associated 
with the Pisot substitutions, in particular it plays a fundamental role in the study of the  Pisot conjecture. 
Geometrical and topological properties of  Rauzy fractals have been studied extensively, 
see among other references~\cite{arnoux-ito,canterini-siegel,holton-zamboni,pytheas_fogg,rauzy,siegel-thuswaldner,sirvent-wang}.
Symmetries in Rauzy fractals were studied in~\cite{sirvent:2}, 
in relation to symmetries that exhibit the symbolic languages 
which define the Rauzy fractals.  
In the present paper we continue the study the symmetric structure of these sets.
We consider the Rauzy fractal of a unimodular irreducible substitution and 
the fractal of its reverse substitution, 
in section~\ref{sec:substitutions} we give definitions of these objects. 
We show that the intersection of these two sets is invariant under reflection through the origin (Corollary~\ref{cor:symmetricfractal}).
Later we show that this set, is obtained by running the balanced pair algorithm of the original substitution and its reverse substitution (Theorem~\ref{thm:main}).

The balanced pair algorithm was introduced by Livshits~(\cite{livshits}) in the context of  the Pisot conjecture, it was also used in ~\cite{sirvent-solomyak} in the same context.
A variant of this algorithm was used later by the first author in~\cite{sellami:1,sellami:2},  
in the study of the intersection of Rauzy fractals associated with 
different substitutions having the same incidence matrix.
This version of the balanced pair algorithm is used in the present article, 
we describe it in section~\ref{sec:bpa}.
The intersection of Rauzy fractals of substitutions having the same incidence matrix,  
has been studied previously in~\cite{sing-sirvent}.

In section~\ref{sec:examples} we present some examples, in particular a well known family of Pisot substitutions (Example 2). 
We describe the balanced pair algorithm in detail for these examples and the intersection of the corresponding Rauzy fractals.
We end the paper with a section of open problems and remarks.


\section{Substitutions and Rauzy fractals}\label{sec:substitutions}
A substitution on a finite alphabet $\A=\{1,\ldots,k\}$ is a map $\s$ from 
$\A$ to the set of finite words in $\A$, i.e., ${\A}^{*}=\cup_{i\geq 0}{\A}^{i}$. 
The map $\sigma$ is extended to $\A^*$ by concatenation, i.e., 
 $\s(\emptyset)=\emptyset$ and $\s(UV)=\s(U)\s(V)$, for all $U$, $V\in\A^*$.
Let $U$ be a word in $\A$, we denote  by $|U|$ the length of $U$. 
We denote by
$[\sigma(i)]_j$ the $j$-th symbol of the word $\sigma(i)$, i.e.,
$\sigma(i)=[\sigma(i)]_1\cdots[\sigma(i)]_{|\sigma(i)|}$. 
 
 \smallskip
 
Let $A^{\N}$ (respectively $A^{\Z}$) denote the set of one-sided 
(respectively two-sided) infinite sequences in $\A$. 
The map $\sigma$, is extended to $\A^{\N}$ and $\A^{\Z}$ in the obvious way:
Let $u=\ldots u_{-1}\dot{u_0}u_1\ldots$ be an element of $\A^{\Z}$, where the dot is used to denote the zeroth position. 
So $\s(u)$ is of the form:
$$
\cdots[\s(u_{-1})]_{1}
\cdots[\s(u_{-1})]_{|\s(u_{-1})|}
[\dot{\s(u_0)}]_1\cdots[\s(u_0)]_{|\s(u_0)|}[\s(u_1)]_1\cdots.
$$

\smallskip

We call $u\in\A^{\N}$ (or  $u\in\A^{\Z}$) a {\em fixed point} of $\s$ 
if $\s(u)=u$ and {\em periodic} if there exists $l>0$ so that it is fixed for $\sigma^l$.

\smallskip

We write $l_i(U)$ for the number of occurrences of the symbol 
$i$ in the word $U$ and denote the vector ${\bf l}(U)=(l_1(U),\ldots,l_k(U))^{t}$. 
The {\em incidence matrix} of the substitution $\s$ is defined as the matrix 
$M_{\s}=M=(m_{ij})$ whose entries $m_{ij}=l_i(\s(j))$ , for $1\leq i,j \leq k$.
Note that $M_{\s}({\bf l}(U))={\bf l}(\s(U))$, for all $U\in\A^*$.
We say the substitution is {\em primitive} if its incidence matrix is 
primitive, i.e., all the entries of  $M^r$ are positive for some $r>0$.

\smallskip

For a primitive substitution there are a finite number of periodic points. 
So we shall assume the substitution has always a fixed point, since we can replace the substitution by a suitable power.
Let $u$ be fixed point of $\sigma$,  
we consider the dynamical system $(\Omega_u,S)$, 
where $S$ is the shift map on $\A^{\N}$ (respectively on $\A^{\Z}$) defined by 
$S(v_0v_1\cdots)=v_1\cdots$ (respectively $S(v)=w$, 
where $w_{i}=v_{i+1}$) and $\Omega_u$ is the closure, in the product topology,  of the orbit 
of the fixed point $u$ under the shift map $S$.

\smallskip

A {\em Pisot number} is a real algebraic integer greater than $1$ such that its Galois conjugates are of norm smaller than $1$.
The Pisot numbers are also known in the literature as {\em Pisot-Vijayaraghavan} or {\em PV numbers}.
We say that a substitution is Pisot  if the Perron-Frobenius eigenvalue of the the incidence matrix is a Pisot number.
A substitution is  {\em irreducible Pisot}  if it is Pisot and the characteristic polynomial of the incidence matrix is irreducible.
An irreducible Pisot substitution is primitive~\cite{canterini-siegel}.
A substitution is {\em unimodular } if  the absolute value of the determinant 
of its incidence matrix is $1$.

\smallskip

There is a long standing conjecture that the dynamical system 
associated with a unimodular irreducible Pisot substitution is measurably  
conjugate to a translation on a $(k-1)$-dimensional torus 
({\em cf.}~\cite{rauzy,solomyak,vershik-livshits}).
This conjecture is known in literature as the Pisot conjecture.
G. Rauzy approached it via geometrical realization of the symbolic system. 
He proved it in the case of the tribonacci substitution, 
$\s(1)=12$, $\s(2)=13$ and $\s(3)=1$ ({\em cf.}~\cite{rauzy}).
In his proof,   the construction of a set in $\R^2$, in general $\R^{k-1}$, plays an important role. 
This set is  known as the Rauzy fractal associated with the substitution.
For references on  conditions under which the Pisot conjecture is true, 
among other references see~\cite{arnoux-ito,barge-diamond,barge-kwapisz,CANT,canterini-siegel,LMS,livshits,pytheas_fogg,rauzy,sing,sirvent-solomyak,sirvent-wang,solomyak,vershik-livshits}.
Before we define Rauzy fractals, we have to introduce some constructions and notation.

\smallskip

Let $\s$ be a unimodular Pisot substitution and 
$\lambda$ the Perron-Frobenius eigenvalue of the incidence matrix $M$, 
so $\lambda$ is a Pisot number.
The characteristic polynomial of $M$ might be reducible, so algebraic degree  of 
 $\lambda$ is smaller or equal that $k$,  the cardinality of the alphabet $\A$.
We decompose $\R^k$ into a direct sum of subspaces, determined by the eigenvualues of $M$.
In particular, we consider:
\begin{itemize}
\item  Let $E^u$ be the {\em expanding space}, i.e. the $\lambda$-expanding space of $M$, 
the eigen-space associated with the eigenvalue $\lambda$.
\item 
Let  $E^s$ be the {\em contracting space}. i.e.  the $\lambda$--contracting  space of $M$, the direct sum of
  the eigen-spaces associated with the Galois conjugates of $\lambda$.
\item  
 Let $E^c$ be the {\em complementary space}, i.e. the direct sum of
 the eigen-spaces associated with the remaining eigenvalues of $M$. 
 \end{itemize}
So, by the definition of the subspaces, we have
 $\R^k=E^u\oplus E^s\oplus E^c$.
 The space $E^c$ is trivial if and only if the substitution is irreducible.
Let $\pi:\R^k\ra E^s$ be the projection of $\R^k$ onto $E^s$ along 
$E^u\oplus E^c$.

\smallskip

Let $u=\ldots u_{-2}u_{-1}u_0u_1u_2\ldots$ be a  fixed point of $\sigma:\A^{\Z}\ra\A^{\Z}$, 
consider the polygonal line or stepped line $(L_n)_{n\in\Z}$ on $\R^k$, given by
$$
L_n:=\left\{\begin{array}{lc}
                \sum_{i=0}^n e_{u_i} &\text{ if } n\geq 0\\
                & \\
                \sum_{i=n-1}^{-1} -e_{u_{i}} &\text{ if } n<0,
                \end{array}\right.
$$
where $\{e_1,\ldots,e_k\}$ is the canonical basis of $\R^k$.
 
We define the {\em Rauzy fractal} associated with $\s$, as
$$
\RRR_{\sigma}:=\overline{\left\{\pi(L_n)\,|\, n\in\Z \right\}}.
$$
We shall also use the notation $\RRR$ for $\RRR_{\sigma}$, whenever the context is clear.
If we consider $\overline{\left\{\pi(L_n)\,|\, n\in\N \right\}}$,
 we obtain the same set, ({\em cf.}~\cite{canterini-siegel});
 see Figure~\ref{fig:projection}.
The construction of the Rauzy fractal, 
does not depend on the selection of the fixed or periodic point of $\sigma$ ({\em cf.}~\cite{ pytheas_fogg}).

\smallskip

The Rauzy fractals are bounded~(\cite{holton-zamboni}), 
they are the closure of their interior~(\cite{sirvent-wang}). 
See~\cite{siegel-thuswaldner}, for a study of different topological properties of these sets. 
The reducible case is studied in~\cite{EIR}.

\begin{figure}
\begin{center}
\scalebox{0.6}{\input{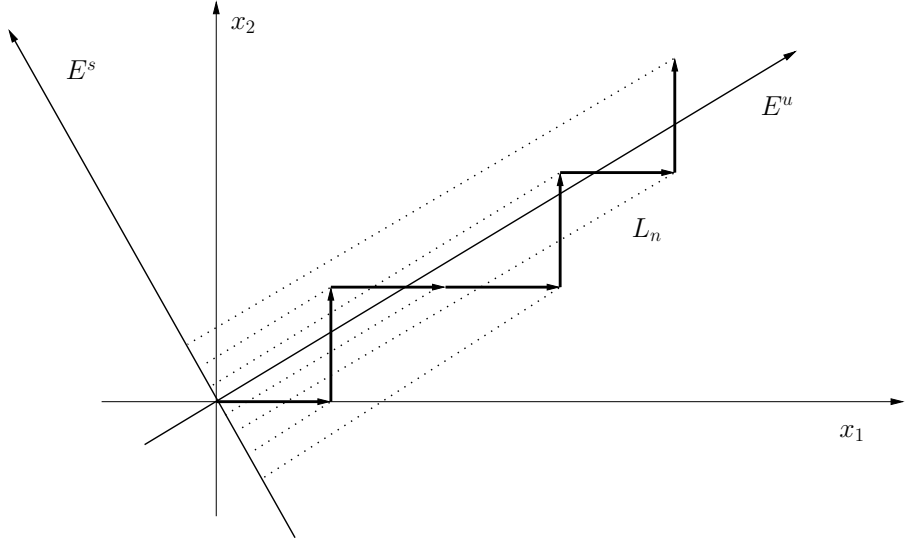} }
\caption{\label{fig:projection}Polygonal line projection in the Rauzy fractal construction for a substitution with $2$ symbols.
}
\end{center}
\end{figure}

\medskip

Let $\hat{\sigma}:\A\ra\A^*$ be the {\em reverse substitution of} $\sigma$, defined as follows: 
$$[\hat{\sigma}(i)]_{j}=[\sigma(i)]_{|\sigma(i)|-j+1}.$$
If $\sigma$  is the tribonacci substitution: $1\ra12$, $2\ra13$, $3\ra 1$, then 
$\hat{\sigma}$ is $1\ra 21$, $2\ra 31$, $3\ra 1$.

\smallskip

\begin{prop}\label{prop:FP}
Let $\sigma$ be a substitution such that  it has a fixed point $u=\ldots u_{-1}u_0u_1\ldots$.
Let  $\hat{\sigma}$ be the reverse substitution of $\s$. 
Then $\hat{\sigma}$ has a fixed point, 
$\hat{u}=\ldots\hat{u}_{-1}\hat{u}_0\hat{u}_1\ldots$, 
with the property $\hat{u}_i=u_{-i-1}$, for $i\in\Z$.
\end{prop}
\begin{proof}
Let $u=\ldots u_{-1}\dot{u_0}u_1\ldots$ be a fixed point of $\s$, where the dot is used to denote the zeroth position. 
So $u$ is of the from:
$$
\cdots[\s(u_{-1})]_{1}
\cdots[\s(u_{-1})]_{|\s(u_{-1})|}
[\dot{\s(u_0)}]_1\cdots[\s(u_0)]_{|\s(u_0)|}[\s(u_1)]_1\cdots.
$$
Let $v\in\A^{\Z}$ defined by $v_i=u_{-i-1}$, for $i\in\Z$, so $v$ is of the form
$$
\cdots[\s(u_0)]_{|\s(u_0)|}\cdots[\s(u_0)]_{1}[\dot{\s(u_{-1})}]_{|\s(u_{-1})|}[\s(u_{-1})|_{|\s(u_{-1})|-1}\cdots
[\s(u_{-1})]_1\cdots.
$$
By the definition of the reverse substitution $\hat{\s}$:
$$
v=\cdots\hat{\s}(u_0)\hat{\s}\dot{(u_{-1})}\hat{\s}(u_{-2})\cdots.
$$
Hence $v=\hat{\s}(v)$, i.e., $v$ is a fixed point for $\hat{\s}$.
\end{proof}

\smallskip

\begin{prop}
Let $\widehat{L}_n$ be the polygonal line associated with 
the reverse substitution of $\sigma$. Then for all integer $n$ we have:
$$
\pi(L_n)=-\pi(\widehat{L}_{-n}).
$$   
\end{prop}
\begin{proof}
By definition, $\widehat{L}_n=\sum_{i=0}^n e_{\hat{u}_i}$, if $n\geq 0$.
Due to Proposition~\ref{prop:FP} 
$$\widehat{L}_n=\sum_{i=0}^n e_{u_{-i-1}}=\sum_{i=-n-1}^{-1} e_{u_{i}}=-L_{-n}.
$$
 Similarly for $n<0$.
Since the projection $\pi$ is linear, we have 
$\pi(L_{-n})=-\pi(\widehat{L}_n)$.
\end{proof}

\smallskip

The incidence matrices of $\sigma$ and $\hat{\sigma}$ are the same, 
so both substitutions have the same spectral properties~({\em cf.}~\cite{queffelec}). 
Therefore we can define the Rauzy fractal associated with $\hat{\sigma}$,
since $\sigma$ is a unimodular irreducible Pisot substitution.
From the previous Proposition follows the next result: 

\begin{cor}\label{cor:symmetricfractal}
Let $\s$ be a unimodular irreducible Pisot substitution 
and  $\hat{\s}$ its reverse substitution. 
Let $\RRR_{\s}$ and $\RRR_{\hat{\s}}$ be the corresponding Rauzy fractals.
Then $\RRR_{\s}=-\RRR_{\hat{\s}}$.   
Moreover
$\RRR_{\s}\cap\RRR_{\hat{\s}}$ is a symmetric set with respect to the origin.
\end{cor}

In Theorem~\ref{thm:main}, we show that if the substitution satisfies some additional and natural hypotheses,
the set $\RRR_{\s}\cap\RRR_{\hat{\s}}$ has non-empty interior.

\subsection{Balanced pair algorithm}\label{sec:bpa}
In this section we introduce the balanced pair algorithm for two substitutions 
$\sigma_1$ and $\sigma_2$ having the same incidence matrix. 
We shall assume that the substitutions are primitive.
This algorithm was introduced in~\cite{sellami:1} and~\cite{sellami:2}, in the context of the study of 
intersection of Rauzy fractals.

\smallskip

Let $U$ and $V$ be two finite words, we say that 
$\begin{pmatrix}
U\\
V\\
\end{pmatrix}$ 
is  a {\em balanced pair} if ${\bf l}(U) = {\bf l}(V)$, 
where ${\bf l}(U)$ is the $k$-dimensional vector that gives the occurrences of the different symbols of the word $U$.

\smallskip

Given a word $U$ we denote by $\langle U\rangle_m$ the proper prefix of $U$ of length $m$.
A {\em minimal balanced  pair} is a balanced pair 
$\begin{pmatrix}
U\\
V\\
\end{pmatrix}$, such that 
${\bf l}(\langle U\rangle_m) \neq {\bf l}(\langle V\rangle_m)$, for $1\leq m < |U|$.

\smallskip

Let $\sigma_1$ and $\sigma_2$ be two irreducible Pisot substitutions with the same incidence matrix.   
Let $u$ and $v$ be the elements of $\A^{\N}$, which are  fixed points of $\s_1$ and $\s_2$, respectively.
We define the balanced pair algorithm associated with the substitutions $\s_1$ and $\s_2$ as follows:

\smallskip

\noindent
We suppose that there exist prefixes $U$ and $V$ of $u$ and $v$, respectively, such that $\begin{pmatrix} U\\ V\end{pmatrix}$ is a minimal balanced pair. 
We call this pair the first minimal balanced pair, of $u$ and $v$.
Under the right hypotheses, considered in section~\ref{sec:intersection}, the first minimal balanced pair always exists.
We apply the substitutions $\s_1$ and $\s_2$ to this balanced pair, in the following manner
$\begin{pmatrix} U\\ V\end{pmatrix}\rightarrow \begin{pmatrix} \sigma_1(U)\\ \sigma_2( V)\end{pmatrix}$. 
Since the substitutions $\s_1$ and $\s_2$ have the same incidence matrix, the pair 
$ \begin{pmatrix} \sigma_1(U)\\ \sigma_2( V)\end{pmatrix}$ is minimal.
We consider this new balanced pair and we decompose it into minimal balanced pairs.
We repeat this procedure to each of this new minimal balanced pairs.
Under the right hypotheses, considered in the next section, 
the set of minimal balanced pairs  is finite, and  the algorithm terminates.

\section{Intersection of Rauzy fractals}\label{sec:intersection}

Let $\sigma_1$ and $\sigma_2$ be two unimodular  irreducible Pisot substitutions with the same incidence matrix.  
We consider their respective  Rauzy fractals  $\RRR_{\sigma_1}$ and $\RRR_{\sigma_2}$.  
Since the origin is always an element of $\RRR_{\s_1}$ and $\RRR_{\s_2}$, 
the intersection of  $\RRR_{\sigma_1}$ and $\RRR_{\sigma_2}$ is non-empty, and it is a compact set because it is intersection of two compacts sets.
Let $\EEE$ be the closure of the intersection of the interior of $\RRR_{\sigma_1}$
and the interior of $\RRR_{\sigma_2}$. 
Through out the article, we shall assume that $0$ is an interior point of one of the Rauzy fractals.

\begin{proposition}
Let $\sigma_1$ and $\sigma_2$ be two unimodular irreducible Pisot substitution with the same incidence matrix. We consider $\RRR_{\sigma_1}$ and $\RRR_{\sigma_2}$ their associated Rauzy fractal. We suppose that $0$ is an inner point to $\RRR_{\sigma_1}$. Then  the set  $\EEE$ has
non-empty interior and strictly positive Lebesgue measure.
\end{proposition}
\begin{proof}
By the assumption that  $0$ is an  inner point of $\RRR_{\sigma_1}$,  
 there exists an open set ${\mathcal U}$ such that $0\in {\mathcal U} \subset \RRR_1$. 
Since the Rauzy fractal is the closure of its interior~(\cite{sirvent-wang}) 
and $0$  is a point of $\RRR_{\sigma_2}$,  there
exists a sequence of points $\{x_n\}_{n\in\N}$ in the interior of
$\RRR_{\sigma_2}$ that converges to $0$. 
Thus there exist open sets ${\mathcal V}_n$ such
that $x_n\in {\mathcal V}_n\subset \RRR_{\sigma_2}$. 
From the fact  $\{x_n\}$
converges to $0$, we conclude that, there exists $N\in\N$  such that $x_N\in {\mathcal U}$.
Therefore  the open set ${\mathcal U}\cap {\mathcal V}_N$ is non-empty and ${\mathcal U}\cap {\mathcal V}_N\subset \RRR_{\sigma_1}\cap\RRR_{\sigma_2}$. 
This implies that  $\EEE$  contains a non-empty open
set; hence it has strictly positive Lebesgue measure.
\end{proof}


If the substitutions $\s_1$ and $\s_2$ satisfy the Pisot conjecture and $0$ is an inner point of one of the Rauzy fractals, then 
the set $\EEE$ is also a Rauzy fractal associated with the substitution defined by the balanced pair algorithm.
This result was proved in~\cite{sellami:2} and we give here an idea of the proof.

\smallskip

\begin{thm}
\label{thm:nonemptyinterior}
Let $\sigma_1$ and $\sigma_2$ be two  unimodular irreducible Pisot substitutions with the same incidence matrix. 
Let $\RRR_{\sigma_1}$ and $\RRR_{\sigma_2}$  be their   two
associated  Rauzy fractals.
Suppose that $0$ is an inner point of  $\RRR_{\sigma_1}$ and both substitutions satisfy the Pisot conjecture. 
We denote by $\EEE$
 the closure of the  intersection  of the interiors of $\RRR_{\sigma_1}$ and $\RRR_{\sigma_2}$. 
 Then $\EEE$ has non-empty
interior, and it is a Rauzy fractal  associated with a Pisot substitution $\Sigma$ on the alphabet of minimal balanced pairs.
\end{thm}

We will assume the following lemma (for the proof see~\cite{sellami:2}), and we give an idea of the proof of Theorem~\ref{thm:nonemptyinterior}.

\begin{lemma}\label{lemma}
Let $\sigma_1$ and $\sigma_2$ be two  unimodular irreducible Pisot substitutions
with the same incidence matrix. Let $\RRR_{\sigma_1}$ and $\RRR_{\sigma_2}$  be their  
associated  Rauzy fractals.
Suppose that ${\sigma_1}$ satisfies the Pisot conjecture and $0$ is an inner point of  $\RRR_{\sigma_1}$.
Let $u$ and $v$ be the one-sided fixed points of $\sigma_1$ and $\sigma_2$ respectively. There exists a finite non-empty set $E$ of minimal balanced pairs, $E = \left\{\begin{pmatrix}U_1\\V_1\end{pmatrix},\ldots, \begin{pmatrix}U_p\\V_p\end{pmatrix}\right\}$, such that 
 the double sequence $\begin{pmatrix}u\\v \end{pmatrix}$ can be decomposed  with elements from $E$.
\end{lemma}

 The intersection set can be obtained as the projection of a fixed point of a new substitution
 defined on the set of minimal balanced pairs. 
Let  $u$ and $v$ be the elements of $\A^{\N}$, such that $\s_1(u)=u$ and $\s_2(v)=v$. 
From the hypotheses that $0$ is an inner point of $\RRR_{\s_1}$ and $\s_1$ satisfies the Pisot conjecture,
 follows that there exist $W_1$ and $W_2$ prefixes of $u$ and $v$ respectively, 
 such that ${\bf l}(W_1)$ and ${\bf l}(W_2)$, i.e.,
 the pair  $\begin{pmatrix}W_1\\W_2\end{pmatrix}$ is a balanced pair, see Lemma 4.2 of~\cite{sellami:2}.
We decompose this balanced pair into minimal balanced pairs. 
We repeat this procedure to each new minimal balanced pair. 
By Lemma~\ref{lemma} the set of minimal balanced pair is finite.
This  follows from the fact the set of common return times is bounded, 
by iteration with $\sigma_1$ and $\sigma_2$, so
we obtain in bounded finite  time the set of all minimal balanced pairs. 

\smallskip

We take the image of each element of the finite  set of minimal balanced pairs.  
The substitution $\Sigma$ is defined as $\Sigma : \begin{pmatrix}U \\  V  \end{pmatrix} \longmapsto \begin{pmatrix}\sigma_1(U)\\ \sigma_2(V)\end{pmatrix}$. The balanced pair $\begin{pmatrix}\sigma_1(U)\\ \sigma_2(V)\end{pmatrix}$ can be decomposed into minimal balanced pairs, and we can write the image of each minimal balanced pair with concatenated minimal balanced pairs. 
So $\Sigma$ is a substitution defined on the set of minimal balanced pairs.
This substitution is Pisot, with the same dominant eigenvalue as $\s_1$~({\em cf.} Lemma 4.5 of~\cite{sellami:2}), however in general it is reducible.
Let $m$ be the number of different minimal balanced pairs, clearly $m\geq k$.
We consider the decomposition of the $m$-dimensional Euclidean space, in the corresponding expanding, contracting  and complementary spaces:
$$
\R^m=E^s_0\oplus E^u_0 \oplus E^c_0.
$$
Let $E^s$ and $E^u$ be the contracting and expanding eigen-spaces corresponding to $\s_1$, since it is irreducible, we have $\R^k=E^u\oplus E^u$.
The substitutions $\Sigma$ and $\s$ have the same dominant eigenvalue, therefore 
$E^s_0=E^s$ and $E^u_0=E^u$.  
Let $\pi:\R^k\ra E^s$ be the projection of $\R^k$   onto $E^s$; and
 $\pi_{\Sigma}:\R^m\ra E^s_0$ be the projection of $\R^m$   onto $E^s_0$. 
 If $\pi':\R^m\ra \R^k$ is the projection of $\R^m$ onto $\R^k$, then $\pi_{\Sigma}=\pi\circ\pi'$.
 
 Let $(L'_n)_{n\in\N}$ be the broken line in $\R^m$ corresponding to a fixed point of the substitution $\Sigma$. 
 And let $(L_n)_{n\in\N}$ be the broken line in $\R^k$ corresponding to the fixed point of $\s_1$.
The points $\pi'(L'_n)$  corresponds to  exactly the common points to the broken lines of $\s_1$ and $\s_2$~({\em cf.}~\cite{sellami:2}). 
So for all $n\geq 0$, there exists $n'\in\N$, such that $\pi'(L'_n)=L_{n'}$, and moreover 
if the point $L_l$ is a point common to the broken lines of $\s_1$ and $\s_2$, then there exists $n_l$ such that $\pi'(L_{n_l})=L_l$.
Hence
$$
\overline{\left\{\pi_{\Sigma}(L'_n)\,:\, n\in\N\right\}} = \RRR_{\s_1}\cap\RRR_{\s_2}.
$$

\smallskip

When we use Theorem~\ref{thm:nonemptyinterior} in the case of the substitutions $\s$ and $\hat{\s}$,  we obtain the following result:

\begin{thm}\label{thm:main}
Let $\s$ be a unimodular irreducible Pisot substitution 
and  $\hat{\s}$ its reverse substitution. 
Let $\RRR_{\s}$ and $\RRR_{\hat{\s}}$ be the respective Rauzy fractals.
We suppose that the substitution $\s$ satisfies the Pisot conjecture
and the origin is an inner point of $\RRR_{\s}$.
Then the set $\RRR_{\s}\cap\RRR_{\hat{\s}}$ has non-empty interior and is a Rauzy fractal
associated with the substitution obtained by balanced pair algorithm of $\s$ and $\hat{\s}$.
\end{thm}

\section{Examples}\label{sec:examples}
In this section we use letters to represent the elements of the alphabet $\A$.

\smallskip

\noindent 
\textbf{Example 1:}\\
We consider the two substitutions $\s_1$ and $\s_2$ defined as:

\begin{center}
$\s_1:\left\{
\begin{array}{ll}
a\rightarrow aba\\
b\rightarrow ab
\end{array}\right.$
\hspace{1cm} and \hspace{1cm} $\s_2:\left\{
\begin{array}{ll}
a\rightarrow aba\\
b\rightarrow ba.
\end{array}\right.$
\end{center}

The Rauzy fractal of $\s_1$ is an interval, so by  Corollary~\ref{cor:symmetricfractal} the Rauzy fractal of $\s_2$ is also an interval.

\smallskip

We  describe the balanced pair algorithm and we obtain a morphism that characterize the common points of these two Rauzy fractals. 
In this example, the first minimal balanced pair that we can consider is the beginning of the two fixed points associated with $\s_1$ and $\s_2$ it will be 
$\begin{pmatrix}
a\\
a\\
\end{pmatrix}$.
We represent the image of the first element of this pair by $\s_1$ and
the second one by $\s_2$. We obtain : 
$\begin{pmatrix}
a\\
a\\
\end{pmatrix}\overset{\s_1, \s_2}\longrightarrow\begin{pmatrix}
aba\\
aba\\
\end{pmatrix}$.
We denote by $A$ the minimal  balanced  pair $\begin{pmatrix}
a\\
a\\
\end{pmatrix}$ and by $B$ the minimal balanced pair $\begin{pmatrix}
b\\
b\\
\end{pmatrix}$.
Hence we  obtain $A\rightarrow ABA.$

\smallskip

The second step is to consider the same process  with the new balanced pair
$B = \begin{pmatrix}
b\\
b\\
\end{pmatrix}$.
We consider the image of this balanced pair with the two substitutions $\s_1$ and
$\s_2$, and we obtain:
\begin{center}
$\begin{pmatrix}
b\\
b\\
\end{pmatrix}\overset{\s_1, \s_2}\longrightarrow\begin{pmatrix}
ab\\
ba\\
\end{pmatrix}$.\\
\end{center}
We obtain an other balanced pair $\begin{pmatrix}
ab\\
ba\\
\end{pmatrix}$
and we denote by $C$  this new balanced pair. We get the image
of $B$ which is $C$.
We continue with this algorithm and we obtain the image of the balanced pair
$\begin{pmatrix}
ab\\
ba\\
\end{pmatrix}$ is the new balanced pair $\begin{pmatrix}
abaab\\
baaba\\
\end{pmatrix}$.
Therefore we obtain that  the image of the letter $C$ is $CAC$.

On total, we  obtain an alphabet  $\mathcal{B}$ on $3$ symbols and we can define the morphism $\Sigma$ as :

\begin{center}
$\Sigma :\left\{
\begin{array}{ll}
A\rightarrow ABA,\\
B\rightarrow C,\\
C\rightarrow CAC.\\
\end{array}\right. $
\end{center}
This morphism $\Sigma$ is the substitution obtained in Theorem~\ref{thm:nonemptyinterior}. 
The characteristic polynomial of the transition matrix of   $\Sigma$ is $(x^2-3x+1)(x-1)$.
The substitution $\Sigma$ generates all the common points of the two Rauzy fractals
associated with  $\s_1$ and $\s_2$.

\medskip

\noindent 
\textbf{Example 2:}\\
In this example we consider the family of Pisot substitutions defined as follows: 

\begin{center}
$\sigma_{1,i}:\left\{
\begin{array}{ll}
a\rightarrow a^ib,\\
b\rightarrow a^ic,\\
c\rightarrow a,
\end{array}\right.$
\hspace{1cm} and \hspace{1cm} $\sigma_{2,i}:\left\{
\begin{array}{ll}
a\rightarrow ba^i,\\
b\rightarrow ca^i,\\
c\rightarrow a.
\end{array}\right.$
\end{center}

\begin{figure}[h]
\begin{center}
\scalebox{0.3}{\includegraphics{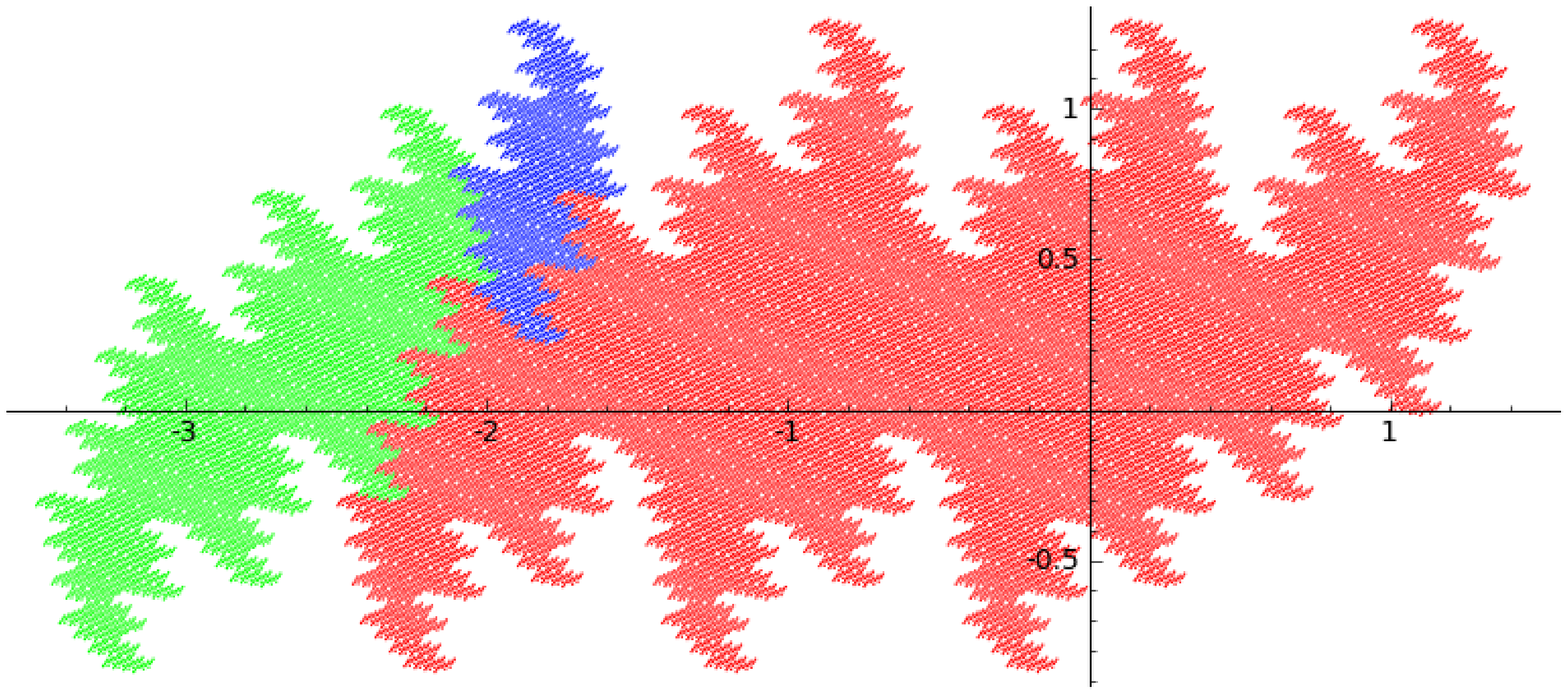}}
\vspace{0.5cm}
\scalebox{0.3}{\includegraphics{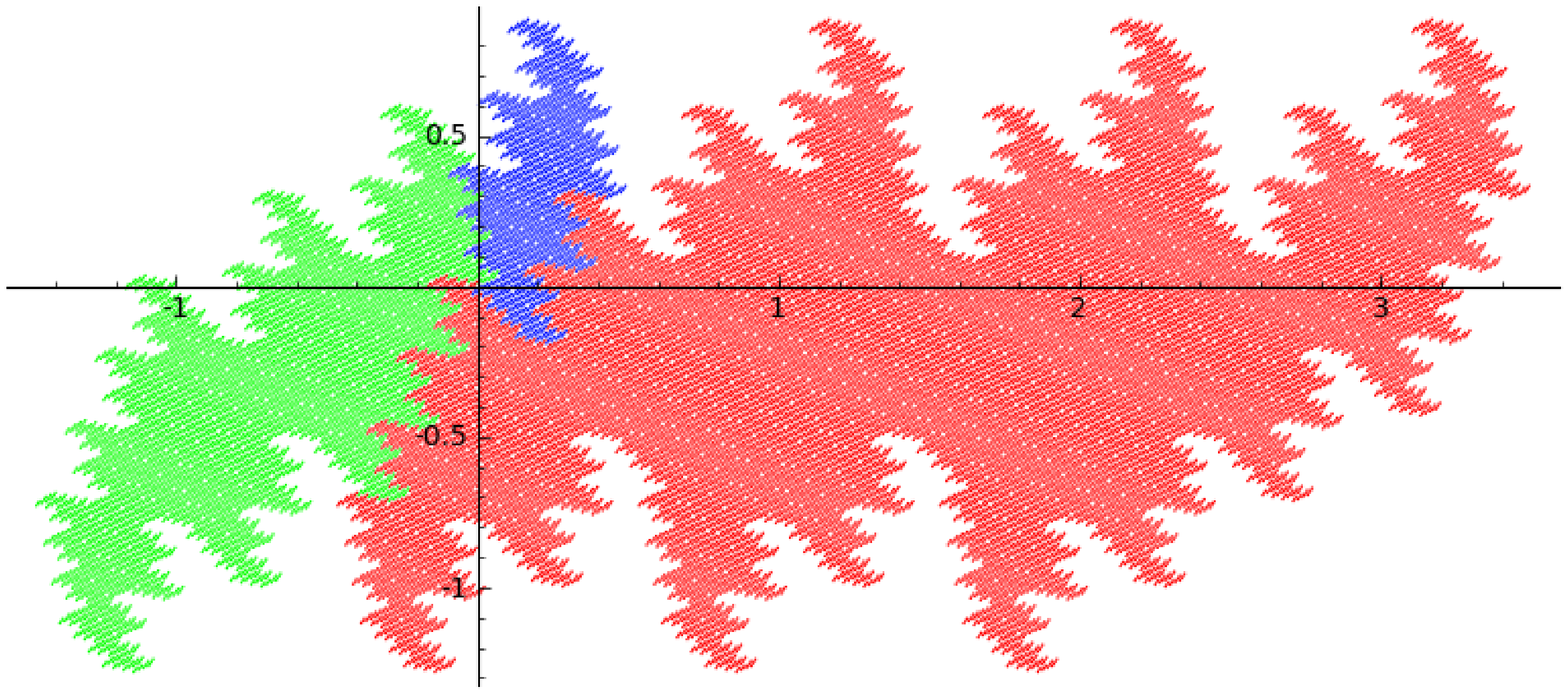}}
\caption{ Rauzy fractals associated with $\sigma_{1,3}$, $\sigma_{2, 3}$. \label{fig:ex2-fractals}}
\end{center}
\end{figure}

Some geometrical and dynamical properties of 
the Rauzy fractals of this family of substitutions have been studied in~\cite{LMST,thuswaldner}.
In particular
the symmetry of these Rauzy fractals  been studied in~\cite{sirvent:2}, 
further properties of these Rauzy fractals were studied in~\cite{NSS}.
They are symmetric, but their center of symmetry is not the origin.
The Rauzy fractals for $\s_{1,3}$ and $\s_{2,3}$ are shown in Figure~\ref{fig:ex2-fractals}.
The classical tribonacci substitution, is $\s_{1,1}$.
The Rauzy fractals of $\s_{1,1}$, $\s_{2,1}$ and their intersections are shown in Figure~\ref{fig:tribo}.  

Since $\s_{2,i}$ is the reverse substitution of $\s_{1,i}$,  both substitutions have the same incidence matrix:
$$M_i = \begin{pmatrix}
i & i & 1 \\
1 & 0 &0 \\
0 & 1 & 0 \\
\end{pmatrix}.
$$  
Let $P_i(x) = x^3 - ix^2- ix-1$ be the characteristic polynomial of $M_i$. The substitutions $\sigma_{1,i}$ and $\sigma_{2,i}$ are unimodular irreducible Pisot substitutions,~({\em cf.}~\cite{brauer}).
 We are interested in this section to study the substitution associated with the intersection. 
In the following proposition we prove that intersection substitution  is defined in an alphabet of six symbols for all $i\geq 1$.

\begin{proposition}
The intersection substitution $\Sigma_i$ associated with  $\sigma_{1,i}$ and $\sigma_{2,i}$ is defined on an alphabet of six symbols as follows:

\begin{center}
$\Sigma_i:\left\{
\begin{array}{ll}
A\rightarrow B,\\
B\rightarrow C,\\
C\rightarrow [(AD)^iAE]^i(AD)^iA,\\
D\rightarrow F,\\
E\rightarrow (AD)^{i-1}A,\\
F\rightarrow [(AD)^iAE]^{i-1}(AD)^iA.
\end{array}\right.$
\end{center}

\end{proposition}
\begin{proof}
We apply the balanced pair algorithm to $\sigma_{1,i}$ and $\sigma_{2,i}$. The first minimal balanced pair is $A = \begin{pmatrix}
a\\
a\\
\end{pmatrix}$. We take the image of $A$ with $\sigma_{1,i}$ and $\sigma_{2,i}$ we obtain a new balanced pair $B =  \begin{pmatrix}
a^i b\\
ba^i\\
\end{pmatrix}$. The balanced pair $B$ is a minimal balanced pair. So we obtain $A\longrightarrow B$.  We take its image again,  we obtain:

$$\Biggl(\begin{matrix}
a^ib\\
ba^i\\
\end{matrix}\Biggr)\overset{\sigma_{1,i}, \sigma_{2,i}}\longrightarrow \Biggl(\begin{matrix}
(a^ib)^ia^ic\\
ca^i(ba^i)^i\\
\end{matrix}\Biggr).$$

\smallskip

We denote by $C = \Biggl(\begin{matrix}
(a^ib)^ia^ic\\
ca^i(ba^i)^i\\
\end{matrix}\Biggr)$, the new balanced pair, it is clear that $C$ is a minimal balanced pair and $B\longrightarrow C$. We continues with the algorithm we calculate the image of $C$, we obtain: 

$$\Biggl(\begin{matrix}
(a^ib)^ia^ic\\
ca^i(ba^i)^i\\
\end{matrix}\Biggr)  \overset{\sigma_{1,i}, \sigma_{2,i}}\longrightarrow \Biggl
(\begin{matrix}
[(a^ib)^ia^ic]^i(a^ib)^ia\\
a(ba^i)^i[ca^i(ba^i)^i]^i\\
\end{matrix}\Biggr). 
$$
Note that the right hand side term can be written as
$$
 \Biggl(\begin{matrix}
[(a^ib) \ldots (a^ib) a^ic]\ldots [(a^ib) \ldots (a^ib) a^ic](a^ib) \ldots (a^ib)a \\
a(ba^)i\ldots (ba^i)[ca^i(ba^i)\ldots (ba^i)]\ldots [ca^i(ba^i)\ldots (ba^i)]\\
\end{matrix}\Biggr).  $$
We can decompose the new balanced pair as follows: 
$$\Biggl[\Biggl(\begin{matrix}
a\\
a\\
\end{matrix}\Biggr)  \Biggl(\begin{matrix}
a^{i-1}b\\
ba^{i-1}\\
\end{matrix}\Biggr)\ldots \Biggl(\begin{matrix}
a\\
a\\
\end{matrix}\Biggr)  \Biggl(\begin{matrix}
a^{i-1}b\\
ba^{i-1}\\
\end{matrix}\Biggr) \Biggl(\begin{matrix}
a\\
a\\
\end{matrix}\Biggr)  \Biggl(\begin{matrix}
a^{i-1}c\\
ca^{i-1}\\
\end{matrix}\Biggr)\Biggr]^i \Biggl[\Biggl(\begin{matrix}
a\\
a\\
\end{matrix}\Biggr)  \Biggl(\begin{matrix}
a^{i-1}b\\
ba^{i-1}\\
\end{matrix}\Biggr)\Biggr]^i.$$
So we obtain two new minimals balanced pairs, we denote $D =  \Biggl(\begin{matrix}
a^{i-1}b\\
ba^{i-1}\\
\end{matrix}\Biggr)$ and $E =  \Biggl(\begin{matrix}
a^{i-1}c\\
ca^{i-1}\\
\end{matrix}\Biggr).$ The image of C is  $  [(AD)^iAE]^i(AD)^iA$. We applies the balanced pair algorithm to $D$ we obtain a new balanced pair $F =  \Biggl(\begin{matrix}
(a^ib)^{i-1}a^ic\\
ca^i(ba^i)^{i-1}\\
\end{matrix}\Biggr)$. Again  $F$ is a minimal balanced pair. 

We continue with minimal balanced pair $E$, we obtain $E\longrightarrow (AD)^iA$ and finally $F\longrightarrow [(AD)^iAE]^i(AD)^iA$.
\end{proof}

\begin{figure}
\begin{center}
\scalebox{0.25}{\includegraphics{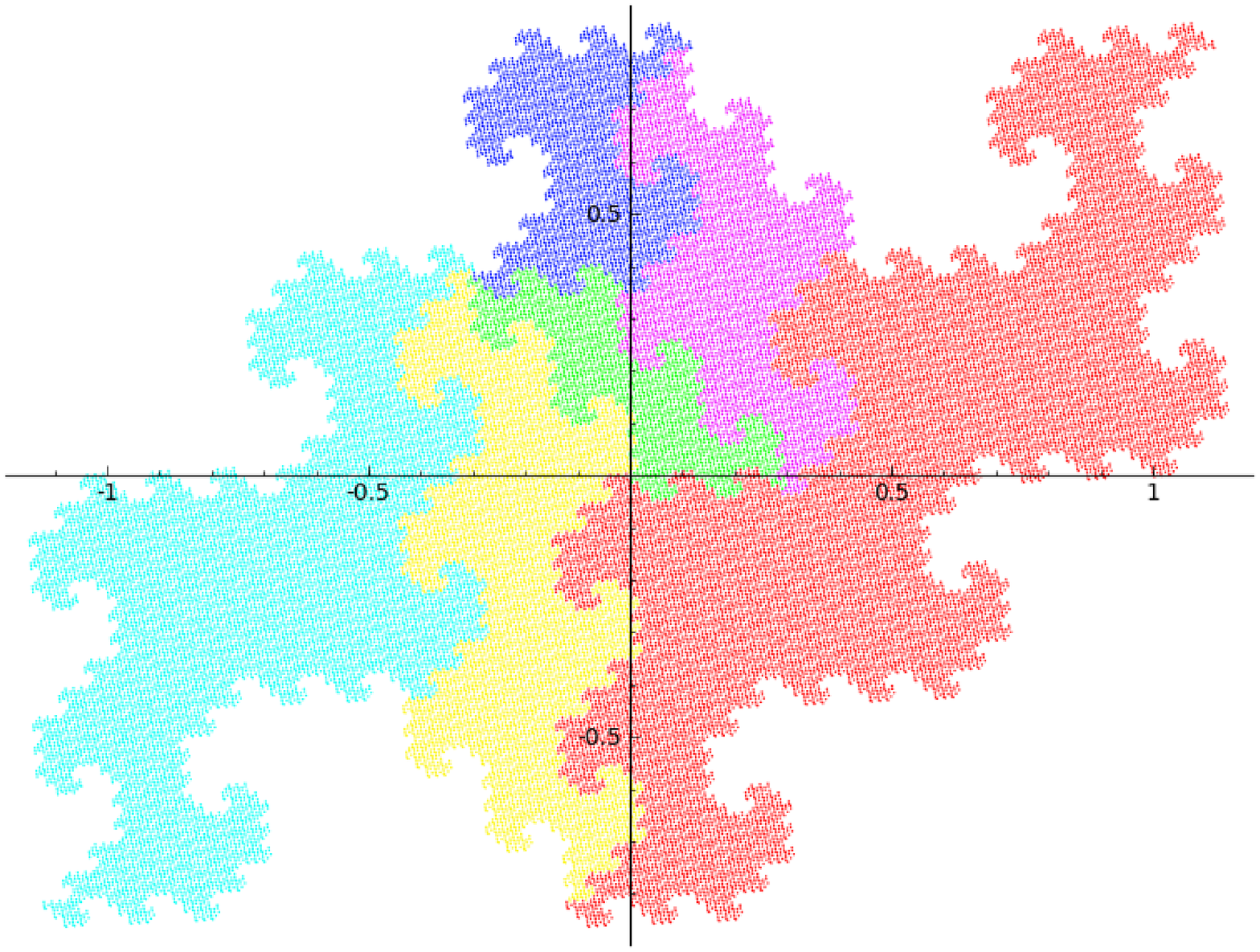}}
\vspace{0.2cm}
\scalebox{0.25}{\includegraphics{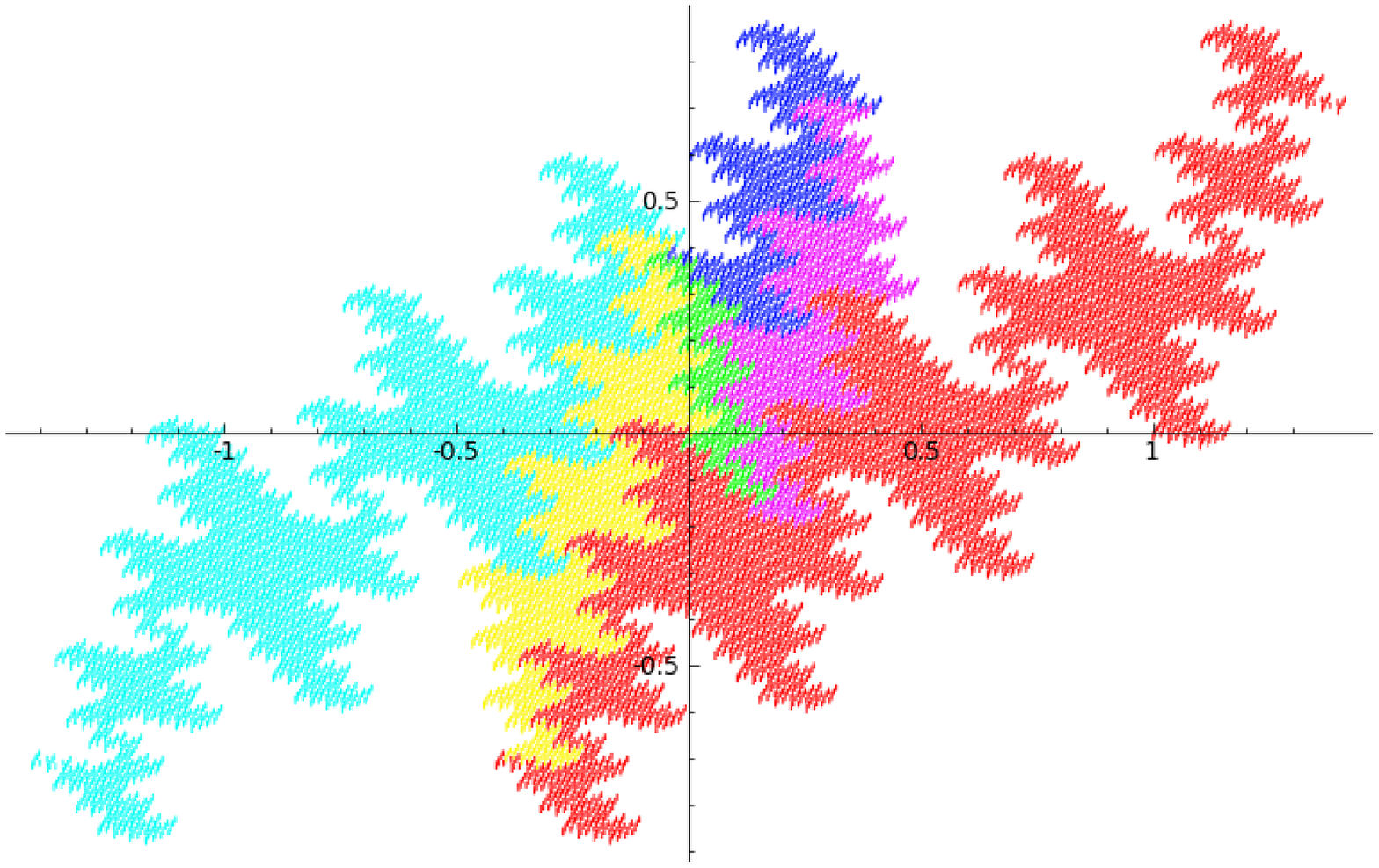}}
\caption{\label{fig:intersection-ex2} Rauzy fractals intersection $\Sigma_{2}$ and $\Sigma_{3}$. }
\end{center}
\end{figure}

The characteristic polynomial of $\Sigma$ is 
$$P_{\Sigma_i}(x) = (x^3 - ix^2- ix-1)(x^3+ix^2+ix-1).$$
Figures~\ref{fig:intersection-ex2} and~\ref{fig:tribo} show the intersection sets for the first three substitutions of this family.\\

\begin{figure}
\begin{center}
\scalebox{0.2}{\includegraphics{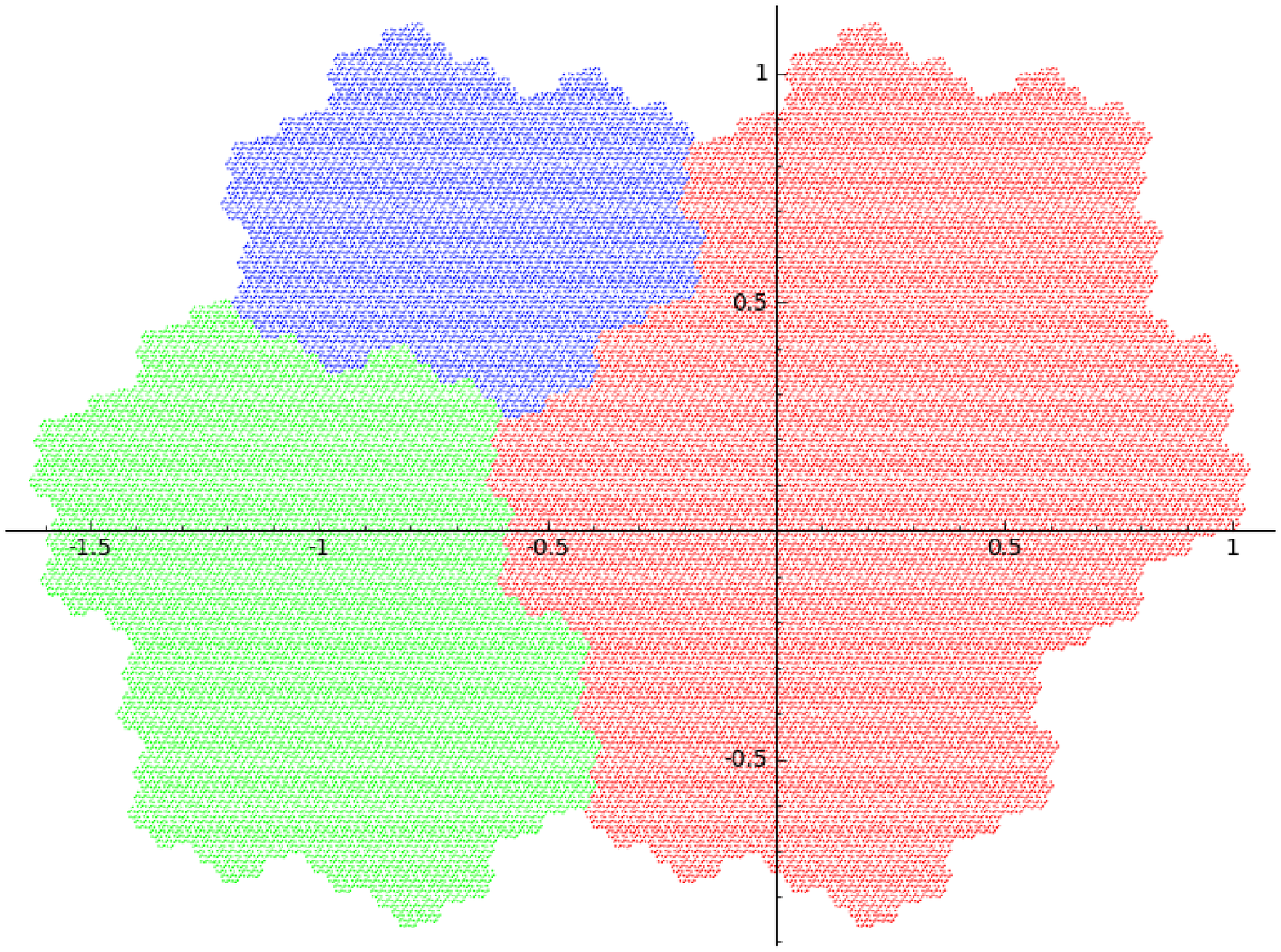}}
\hspace{0.2cm}
\scalebox{0.2}{\includegraphics{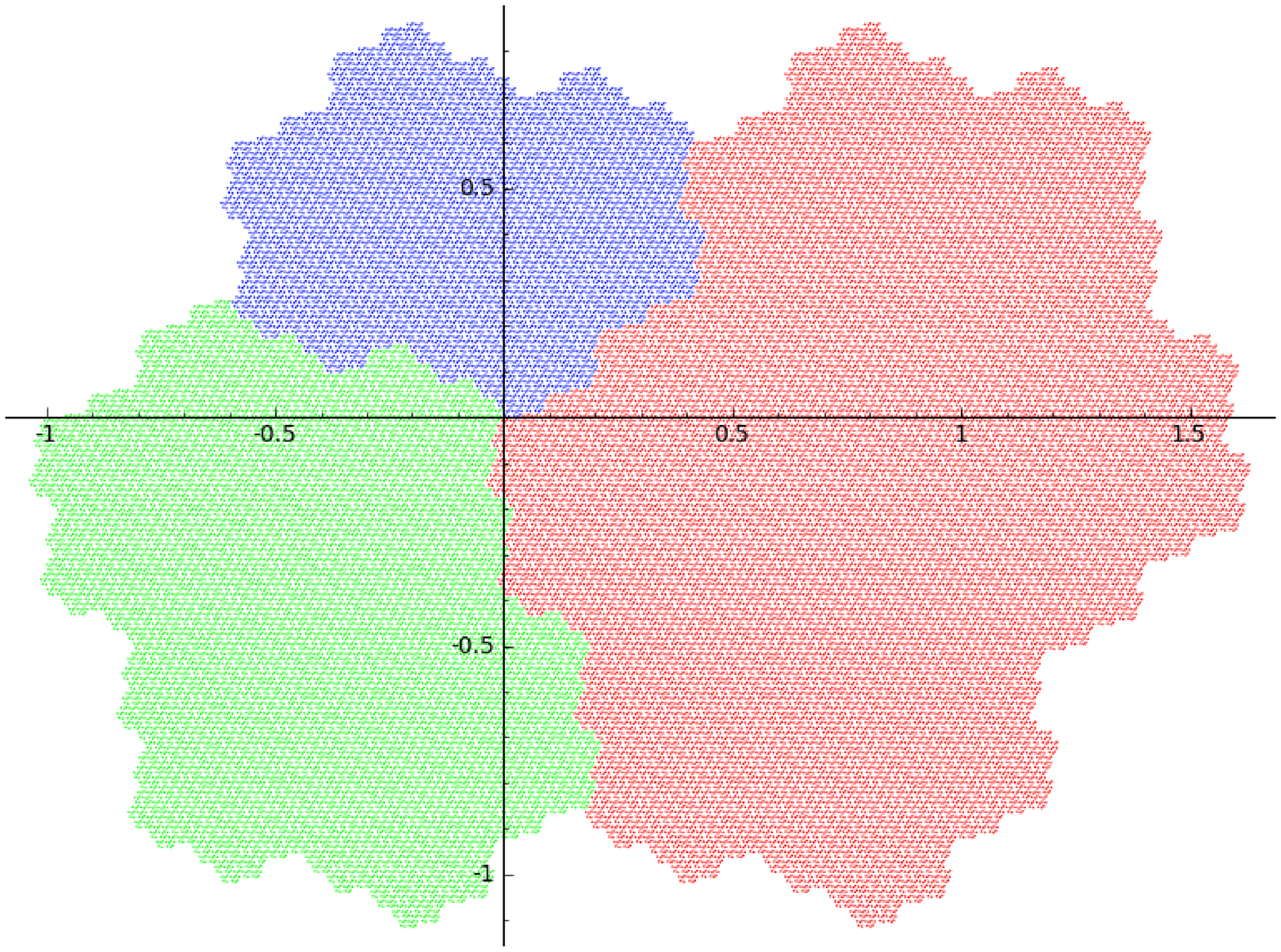}}
\vspace{0.25cm}
\scalebox{0.2}{\includegraphics{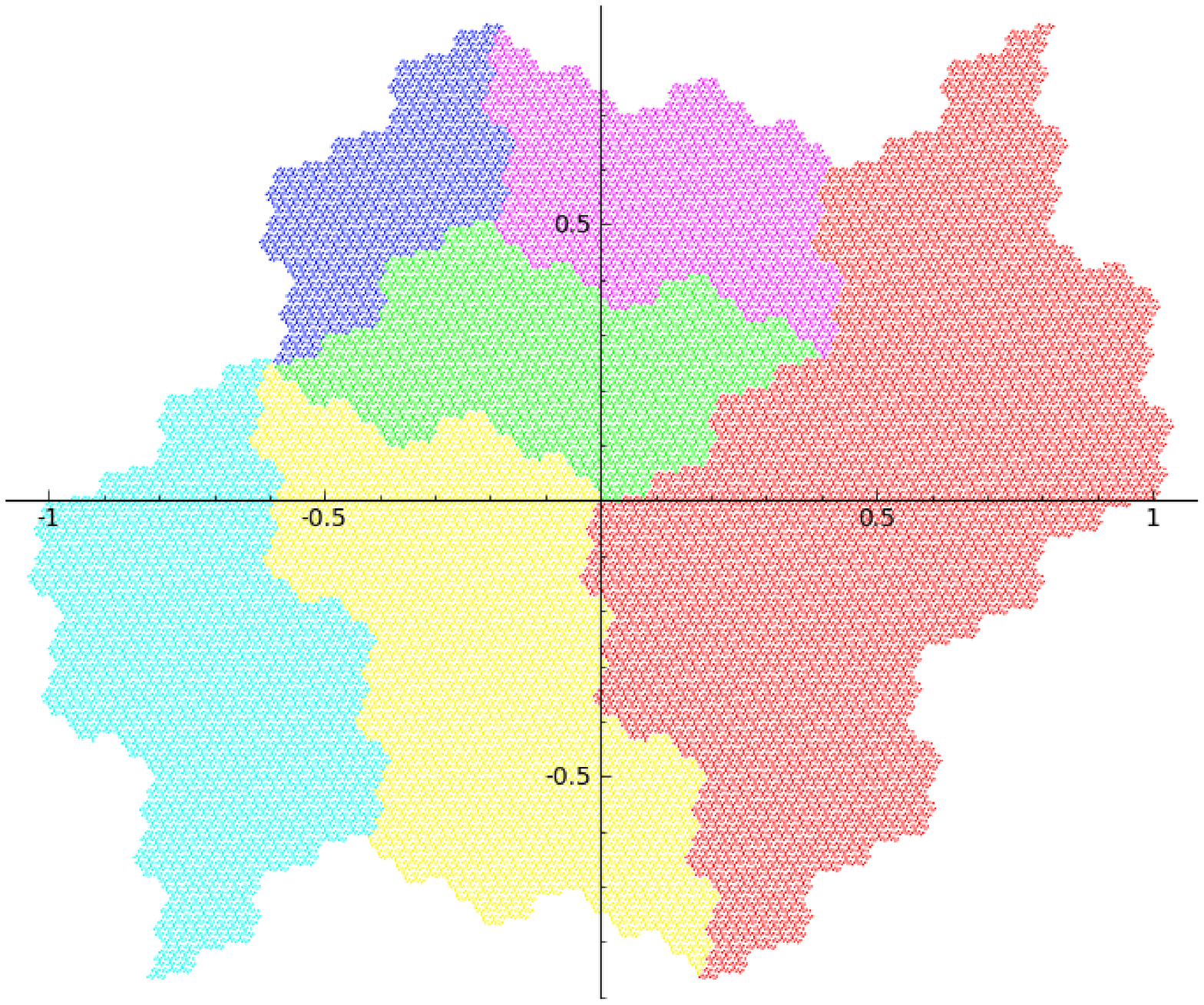}}
\caption{\label{fig:tribo} Rauzy fractals of $\sigma_{1,1}$, $\sigma_{2,1}$ and $\Sigma_1$. }
\end{center}
\end{figure}

\medskip

\noindent 
\textbf{Example 3:}\\
In this example we consider the two  substitutions defined as follows: 
\begin{center}
$\sigma_1:\left\{
\begin{array}{ll}
a\rightarrow ab,\\
b\rightarrow ca,\\
c\rightarrow a,
\end{array}\right.$
\hspace{1cm} and \hspace{1cm} $\sigma_2:\left\{
\begin{array}{ll}
a\rightarrow ba,\\
b\rightarrow ac,\\
c\rightarrow a.
\end{array}\right.$
\end{center}
The substitution $\s_1$ is known as the flipped tribonacci substitution~(\cite{sirvent:1}).
When we apply the balanced pair algorithm to these two substitution, 
we obtain the substitution $\Sigma$ for intersection on $15$ symbols  defined as:
$$
\begin{array}{lllll}
A\rightarrow B, & B\rightarrow ACA, & C\rightarrow D, & D\rightarrow E, & E\rightarrow AFA,\\
F\rightarrow DGHGD, & G\rightarrow I, & H\rightarrow JKJ, & I\rightarrow J, & J\rightarrow ALA,\\
K\rightarrow AMAMA, & L\rightarrow DGD, & M\rightarrow N, & N\rightarrow AOA, & O\rightarrow AMAMAMA.
\end{array}
$$

The characteristic polynomial of the transition matrix of $\Sigma$ is
$$
(x - 1)  (x + 1)  (x^2 - x + 1)  (x^3 - x^2 - x - 1) 
 (x^3 + x^2 + x- 1)  (x^5 + x^4 - 2x^2 - 3x + 1).
$$
The Rauzy fractal associated with $\Sigma$, i.e., the intersection of the Rauzy fractals of $\s_1$ and $\s_2$ is shown in Figure~\ref{fig:ex3}.

\begin{figure}
\begin{center}
\scalebox{0.2}{\includegraphics{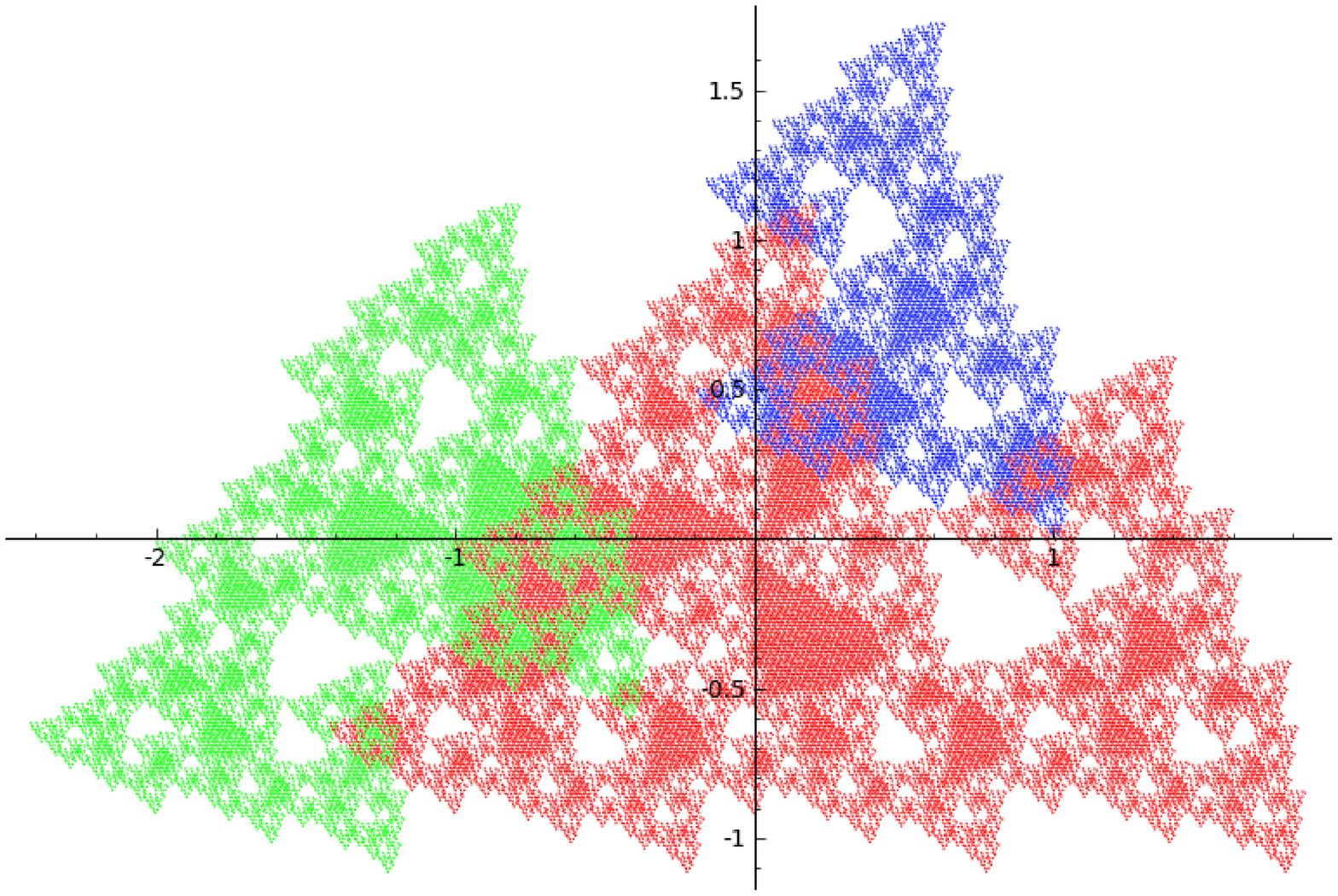}}
\hspace{0.2cm}
\scalebox{0.2}{\includegraphics{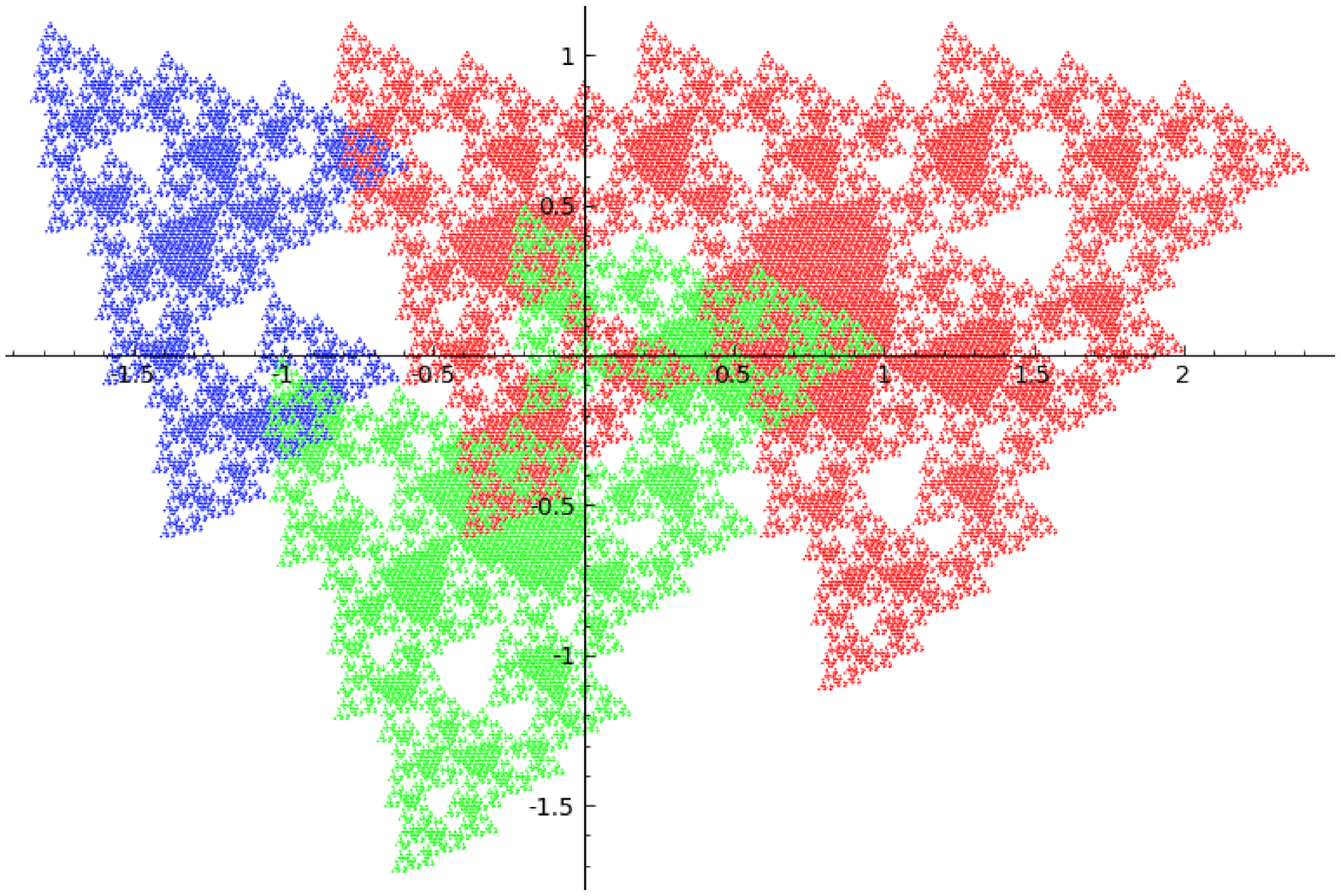}}
\vspace{0.25cm}
\scalebox{0.3}{\includegraphics{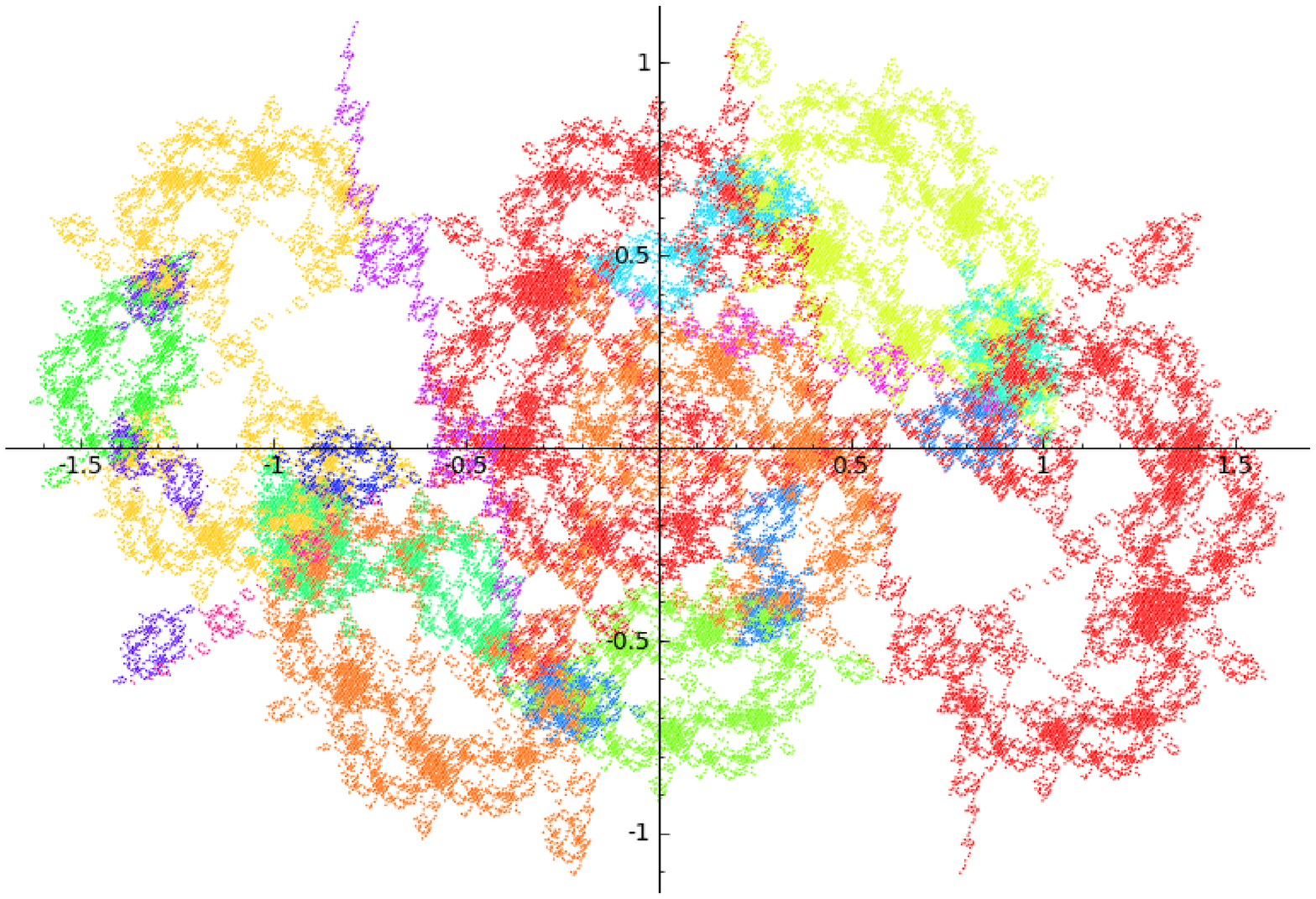}}
\caption{\label{fig:ex3} Rauzy fractals of Example 3 and their intersection. }
\end{center}
\end{figure}

\medskip

\noindent 
\textbf{Example 4:}\\
In the previous examples, we have seen that the substitution $\Sigma$ has the property: $\Sigma(U)$ is a palindrome for each symbol $U$ in the alphabet where $\Sigma$ is defined.  
Here we present an example in which this situation does not occur.
Consider the substitutions in two symbols, given by
\begin{center}
$\s_1:\left\{
\begin{array}{ll}
a\rightarrow aabbaabab\\
b\rightarrow ab
\end{array}\right.$
\hspace{1cm} and \hspace{1cm} $\s_2:\left\{
\begin{array}{ll}
a\rightarrow babaabbaa\\
b\rightarrow ba.
\end{array}\right.$
\end{center}
We remark that these substitutions are unimodular irreducible Pisot. 
When we run the balanced pair algorithm, we get the following balanced pairs: 
$$
\Biggl(\begin{matrix}
aabb\\
baba\\
\end{matrix}\Biggr),
\Biggl(\begin{matrix}
abab\\
bbaa\\
\end{matrix}\Biggr),
\Biggl(\begin{matrix}
aab\\
baa\\
\end{matrix}\Biggr),
\Biggl(\begin{matrix}
abaabb\\
bbaaba\\
\end{matrix}\Biggr),
\Biggl(\begin{matrix}
aabab\\
babaa\\
\end{matrix}\Biggr).
$$
If we denote them by $A, B,C,D, E$, respectively.
The resulting substitution on this alphabet is:
$$
\Sigma:\left\{
\begin{array}{l}
A\ra ACDEB,\\
B\ra AEDCB,\\
C\ra ACDCB,\\
D\ra AEDCDEB,\\
E\ra ACDEDCB.
\end{array}
\right.
$$
It can be observed that $\Sigma(U)$ is not a palindrome, for any $U$ in the alphabet.

The characteristic polynomial of the matrix associated with $\Sigma$ is
$$
x^2(x-1)(x^2-6x+1).
$$

\medskip

\noindent 
\textbf{Example 5:}\\
Let  $\s$ be the substitution on three symbols defined by
$$
\begin{array}{ccc}
a\ra abc, & b\ra a, & c\ra ac.
\end{array}
$$
It has a unique one-sided fixed point:
$$u:=\text{``}\s^{\infty}(a)\text{"}=abcaacabcabcacabcaacabca\cdots.$$
We conjecture the origin is a boundary point of its Rauzy fractal.

Let $\hat{\s}$ be its reversed substitution. It has a unique one-sided fixed point:
$$
v:=\text{``}\hat{\s}^{\infty}(c)\text{"}=cacbcaacbacacbacbacaacba\ldots.
$$

 We conjecture that there is no initial balanced pair between $u$ and $v$, i.e.
 the balanced pair algorithm it cannot be applied to $u$ and $v$. 
 If the conjecture is true, then the intersection of both Rauzy fractals is reduced to the origin.
 Figure~\ref{fig:ex5} shows the Rauzy fractals of $\s$ and $\hat{\s}$.
 
 \begin{figure}
\begin{center}
\scalebox{0.2}{\includegraphics{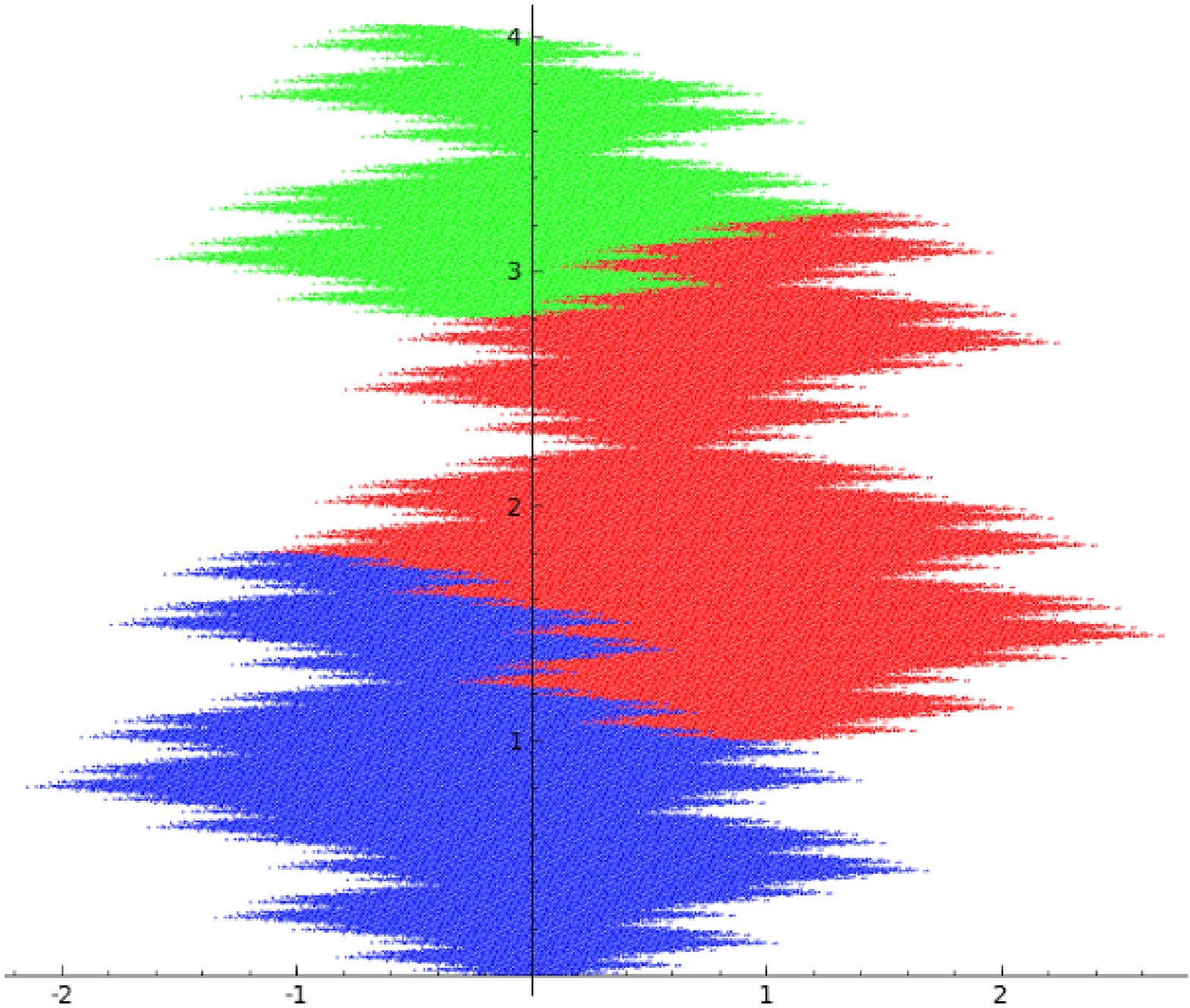}}
\hspace{0.2cm}
\scalebox{0.2}{\includegraphics{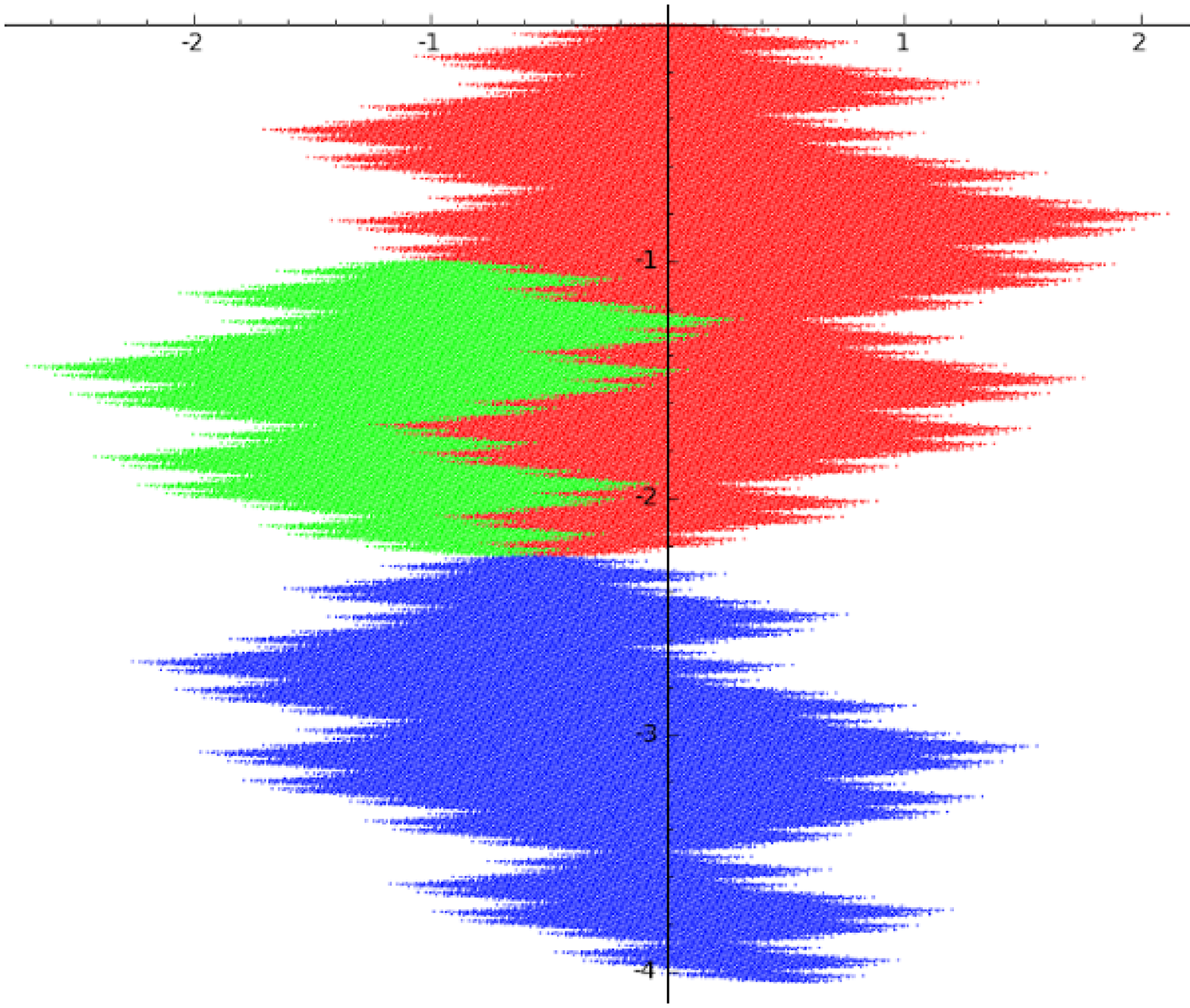}}
\caption{\label{fig:ex5} Rauzy fractals of Example 5.}
\end{center}
\end{figure}


\section{Remarks and open questions}\label{sec:remarks}
\begin{enumerate}
\item 
We say that a substitution $\s$ satisfies the strong coincidence condition on prefixes (respectively on suffixes) if for any two symbols $i,j\in\A$ then there exists $n\in\N$, $a\in\A$ and $p,q,r,t\in\A^*$, such that
$$
\begin{array}{lr}
\s^n(i)=pat \text{ and } \s^n(j)=qar,  & \text{ with } {\bf l}(p)={\bf l}(q)
\end{array}
$$
$$
\text{ (respectively } {\bf l}(t)={\bf l}(r) \text{).} 
$$
Every irreducible unimodular Pisot substitution in two symbols satisfies the strong coincidence condition~({\em cf.}~\cite{barge-diamond}).
It is conjectured that all irreducible unimodular Pisot substitutions satisfies the strong coincidence condition. 

Clearly a substitution  $\s$ satisfies the strong coincidence condition on prefixes (or suffixes) if and only if $\hat{\s}$ satisfies the strong coincidence condition on suffixes (or prefixes).

If $\s^n$ has a unique fixed point in $\A^{\N}$ for all $n\geq 1$ then it satisfies the strong coincidence condition.

\item The Rauzy fractal $\RRR_{\s}$ admits a partition
$\{\RRR_{\s}(1),\ldots,\RRR_{\s}(k)\}$, where 
$$
\RRR_{\s}(j):=\overline{\left\{\pi(L_n)\,|\, u_n=j \text{ and } n\in\N \right\}}.
$$
This partition is called the {\em natural decomposition} of $\RRR_{\s}$.
The set $\RRR_{\s}$ and its natural decomposition can be obtained as 
a fixed point of a graph directed iterated function system, for details see~\cite{sirvent-wang}.
If the substitution $\s$ satisfies the strong coincidence condition,
then  the sets
$\RRR_{\s}(j)$ are measure-wise disjoint~(\cite{arnoux-ito}). 

We note that the natural decomposition of  $\RRR_{\hat{\s}}$ does not have to be 
the reflection through the origin of the natural decomposition of $\RRR_{\s}$,
as it can be seen in Figures~\ref{fig:ex2-fractals} and~\ref{fig:tribo}, 
for some substitutions considered in Example 2.

\item Is it possible to generalize the construction described in this article in the reducible and/or non-unimodular case?
\item We wonder if it is possible to obtain subsets of Rauzy fractals with other symmetries, via the balanced pair algorithm.
\item Is the Hausdorff dimension of the boundary of the intersection set the same of the dimension of the boundary of the Rauzy fractal?
\item Let $p(x)$ be the characteristic polynomial of the matrix $M_{\s}$ and $q(x)$ its reciprocal.
In  examples 2 and 3, $p(x)$ and $q(x)$ are factors of
the characteristic polynomial of the matrix $M_{\Sigma}$, 
where $\Sigma$ is the substitution obtained by the balanced pair algorithm of $\s$ and $\hat{\s}$. 
In examples 1 and 4 we have that $p(x)=q(x)$.
We conjecture that if $p(x)\neq q(x)$ then $p(x)$ and $q(x)$ are factors of
the characteristic polynomial of the matrix $M_{\Sigma}$, when $\s$ is  unimodular irreducible Pisot.
\end{enumerate}

\end{document}